\documentclass[11pt,leqno]{article}

\usepackage{amsmath,amssymb,amsthm}
\usepackage[all]{xy}
\usepackage{lscape}

\makeatletter

\newtheorem*{Theorem}{Theorem~\ref{main.thm}}

\newtheorem{theorem}{Theorem}[section]
\newtheorem{proposition}[theorem]{Proposition}
\newtheorem{lemma}[theorem]{Lemma}
\newtheorem{corollary}[theorem]{Corollary}

\theoremstyle{definition}

\theoremstyle{remark}

\theoremstyle{remark}
\newtheorem{remark}[theorem]{Remark}

\def\({{\rm (}}
\def\){{\rm )}}

\let\Mathrm\operator@font
\let\Cal\mathcal
\let\Bbb\mathbb
\newcommand{\fm}{\ensuremath{\mathfrak m}}

\def\standop#1{\mathop{\Mathrm #1}\nolimits}
\def\difstop#1#2{\expandafter\def\csname #1\endcsname{\standop{#2}}}
\def\defstop#1{\difstop{#1}{#1}}

\defstop{AB}
\defstop{ann}
\defstop{Ass}
\defstop{Add}
\defstop{Alt}
\defstop{Ass}
\difstop{adim}{AFHdim}

\defstop{Bl}

\defstop{CMFI}
\defstop{codim}
\defstop{Coh}
\defstop{coht}
\defstop{Coker}
\defstop{Cone}
\defstop{Cl}
\defstop{Cox}
\defstop{cosk}
\defstop{cd}
\defstop{cmd}
\difstop{charac}{char}

\defstop{depth}
\defstop{Der}
\defstop{Div}
\defstop{div}

\defstop{EM}
\defstop{embdim}
\defstop{End}
\defstop{ev}
\defstop{Ext}

\difstop{fdim}{flat.dim}
\defstop{Flat}
\defstop{Func}
\defstop{Fpqc}

\def\GL{\text{\sl{GL}}}
\defstop{GD}
\defstop{Good}
\defstop{Gal}
\defstop{Grass}

\difstop{height}{ht}
\defstop{Hom}
\defstop{Hy}

\def\id{\mathord{\Mathrm{id}}}
\def\Id{\mathord{\Mathrm{Id}}}
\difstop{Image}{Im}
\defstop{ind}
\defstop{ini}

\defstop{Ker}

\defstop{Lch}
\defstop{length}
\defstop{Lin}
\defstop{Lqc}
\defstop{lqc}
\defstop{LQ}
\defstop{LM}
\defstop{Lie}
\defstop{Loc}

\defstop{Mat}
\defstop{Max}
\defstop{Min}
\defstop{Mod}
\defstop{Mor}
\defstop{MCM}
\defstop{Map}
\difstop{mred}{red}

\defstop{Nerve}
\defstop{NonSFR}
\defstop{NonCMFI}
\defstop{NonNor}
\defstop{Nor}

\defstop{ob}
\defstop{Ob}

\defstop{PA}
\defstop{PM}
\defstop{PR}
\defstop{Proj}
\defstop{Prin}
\defstop{Pic}

\defstop{Qch}
\defstop{qch}

\defstop{rad}
\defstop{rank}
\def\red{_{\Mathrm{red}}}

\defstop{res}
\defstop{Reg}
\defstop{Ref}
\defstop{Rat}

\def\SL{\text{\sl{SL}}}

\defstop{Spec}
\defstop{supp}
\defstop{Supp}
\defstop{Sym}
\defstop{Sing}
\defstop{SFR}
\defstop{Soc}
\defstop{Sh}

\difstop{tdeg}{trans.deg}

\defstop{Tor}
\defstop{TRD}
\difstop{trace}{tr}
\defstop{tot}

\defstop{Zar}
\defstop{FL}
\defstop{add}
\defstop{Ind}

\difstop{summ}{sum}
\def\eHK{e_{\mathrm{HK}}}

\def\aq{/\!\!\hspace{.2ex}/}

\def\O{\Cal O}

\def\fm{\mathfrak{m}}

\def\QQ{{\mathbb{Q}}}
\def\RR{{\mathbb{R}}}
\def\ZZ{{\mathbb{Z}}}

\def\sdarrow#1{\downarrow\hbox to 0pt{\scriptsize$#1$\hss}}
\def\suarrow#1{\uparrow\hbox to 0pt{\scriptsize$#1$\hss}}
\def\ssearrow#1{\searrow\hbox to 0pt{\scriptsize$#1$\hss}}

\def\norm#1{\lVert #1\rVert}
\def\mt#1{$#1$}

\def\section{\@startsection{section}{1}{\z@ }%
  {-3.5ex plus -1ex minus -.2ex}{2.3ex plus .2ex}{\bf }}

\long\def\refname{\par\kern -3ex
  \begin{center}\rm R\sc{eferences}\end{center}\par\kern 
  -2ex}

\def\@seccntformat#1{\csname the#1\endcsname.\quad}

\def\@@@sect#1#2#3#4#5#6[#7]#8{%
  \ifnum #2>\c@secnumdepth 
  \def \@svsec {}\else \refstepcounter {#1}%
  \def\@svsec{}
  \fi 
  \@tempskipa #5\relax 
  \ifdim \@tempskipa >\z@ 
  \begingroup #6\relax \@hangfrom {\hskip #3\relax 
    \@svsec}{\interlinepenalty \@M #8\par }\endgroup 
  \csname #1mark\endcsname {#7}
  \else 
  \def \@svsechd {#6\hskip #3\@svsec #8\csname #1mark\endcsname {#7}}
  \fi \@xsect {#5}}

\def\@@@startsection#1#2#3#4#5#6{%
  \if@noskipsec \leavevmode \fi \par \@tempskipa #4\relax \@afterindenttrue 
  \ifdim \@tempskipa <\z@ \@tempskipa -\@tempskipa \@afterindentfalse 
  \fi \if@nobreak \everypar {}\else \addpenalty {\@secpenalty }\addvspace 
  {\@tempskipa }\fi \@ifstar {\@ssect {#3}{#4}{#5}{#6}}{\@dblarg 
    {\@@@sect {#1}{#2}{#3}{#4}{#5}{#6}}}}

\def\theparagraph{\thesection.\arabic{paragraph}}
\def\aparagraph{\@@@startsection{paragraph}{2}{\z@ }%
  {1.75ex plus .2ex minus .15ex}{-1em}{\bf(\theparagraph) } }
\def\paragraph{\@@@startsection{paragraph}{2}{\z@ }%
  {1.75ex plus .2ex minus .15ex}{-1em}{}{\bf(\theparagraph)} }

\let\c@theorem\c@paragraph

\advance\textheight 12mm
\advance\textwidth 8mm
\title{Generalized $F$-signatures of the rings of invariants of finite group schemes}
\author{{\sc M}itsuyasu {\sc H}ashimoto\thanks{Partially supported by JSPS KAKENHI Grant number 20K03538
    and MEXT Promotion of Distinctive Joint Research Center Program JPMXP0619217849.}
  \and
      {\sc F}umiya {\sc K}obayashi}
\date{}

\begin{document}

\maketitle

\begin{abstract}
  Let $k$ be a perfect field of prime characteristic $p$,
  $G$ a finite group scheme over $k$, and $V$ a finite-dimensional $G$-module.
  Let $S=\mathop{\mathrm{Sym}}V$ be the symmetric algebra with the standard grading.
  Let $M$ be a $\Bbb Q$-graded $S$-finite $S$-free $(G,S)$-module, and
  $L$ be its $S$-reflexive graded $(G,S)$-submodule.
  Assume that the action of $G$ on $V$ is small in the sense that there exists some
  $G$-stable Zariski closed subset $F$ of $V$ of codimension two or more such that
  the action of $G$ on $V\setminus F$ is free. Generalizing the result of P.~Symonds
  and the first author, we describe the Frobenius limit $\mathop{\mathrm{FL}}(L^G)$
  of the $S^G$-module $L^G$.
  In particular, we determine the generalized $F$-signature $s(M,S^G)$ for each indecomposable
  gradable reflexive $S^G$-module $M$.
  In particular, we prove the fact that the $F$-signature $s(S^G)=s(S^G,S^G)$ equals $1/\dim k[G]$ if $G$ is linearly
  reductive (already proved by Watanabe--Yoshida, Carvajal-Rojas--Schwede--Tucker, and Carvajal-Rojas)
  and $0$ otherwise (some important cases has already
  been proved by Broer, Yasuda, Liedtke--Martin--Matsumoto).
\end{abstract}

\section{Introduction}

Study of Frobenius maps and their iterations has long been important in
studying Noetherian commutative rings of prime characteristic $p$ \cite{Huneke}.

For simplicity, let $p$ be a prime number, $k$ a perfect field of characteristic $p$,
and $A$ be a complete Noetherian local ring of characteristic $p$ whose residue
field is $k$ (in the graded case, let $A=\bigoplus_{n\geq 0}A_n$ be a finitely
generated positively graded $k$-algebra with $A_0=k$) in this introduction.

Huneke and Leuschke \cite{HL} defined the $F$-signature $s(A)$ of $A$ using the
iteration $F_A^e:A\rightarrow {}^eA$ of the Frobenius map, where ${}^eA=A$.
We can decompose ${}^eA=A^{a_e}\oplus M_e$ as an $A$-module, where $M_e$ does not have a free summand,
and $s(A)$ is defined to be $\lim_{e\rightarrow \infty}a_e/p^{de}$, where $d=\dim A$.
Tucker \cite{Tucker} proved that the limit exists and $s(A)$ is well-defined.
As Yao \cite{Yao} pointed out, the $F$-signature agrees with the minimal relative
Hilbert--Kunz multiplicity defined by Watanabe and Yoshida \cite{WY}.
$A$ is regular if and only if $s(A)=1$ \cite{HL}, \cite{WY}.
Moreover, $s(A)>0$ if and only if $A$ is strongly $F$-regular \cite{AL}.

Let $G$ be a finite group and $V$ a $d$-dimensional representation of $G$ over $k$.
Let $S=\Sym V$ be the symmetric algebra, and $A=S^G$.
We say that the action of $G$ on $V$ is small if $G\rightarrow \GL(V)$ is
injective, and its image does not have a pseudo-reflection.
Watanabe and Yoshida proved that if the action is small and $G$ is linearly reductive
(or equivalently, $p$ does not divide the order $|G|$ of $G$), then
$s(A)=1/|G|$.
From the result of Broer \cite{Broer} and Yasuda \cite{Yasuda}, $s(A)=0$ if
$G$ is small but not linearly reductive.

It is natural to ask the asymptotic behavior of other indecomposable summands than the free one of ${}^eA$.
The first author and Yusuke Nakajima \cite{HN}
named it the generalized $F$-signature, and calculated them for the invariant subring $A$
for the case that $G$ is small and linearly reductive \cite{HN}.
After that, the first author and P.~Symonds \cite{HS} defined the
Frobenius limit $\FL([A])=\lim_{e\rightarrow\infty}[{}^eA]/p^{de}$
of $[A]$, where the limit is taken in certain normed space whose $\RR$-basis is the set of isomorphism classes of
indecomposable $\QQ$-graded $A$-modules (up to shiftings), see (\ref{normed.par}).
They proved that
\[
\FL([A])=\frac{1}{|G|}[S]=\frac{1}{|G|}\sum_{i=1}^r\frac{\dim_k V_i}{\dim_k \End_G V_i}[M_i],
\]
where $V_1,\ldots,V_r$ is the complete set of representatives of the isomorphism classes of
simple $G$-modules, $P_i$ the projective cover of $V_i$, and $M_i=(P_i\otimes_k S)^G$ \cite[Theorem~5.1]{HS}.
This information is enough to deduce that the generalized $F$-signature
$s(M_i,A)=\lim_{e\rightarrow\infty}c_{i,e}/p^{de}$ agrees with $\frac{\dim_k V_i}{|G|\dim_k \End_G V_i}$,
where ${}^eA=M_i^{c_{i,e}}\oplus N_{i,e}$ such that $N_{i,e}$ does not have $M_i$ as a direct summand.
If $k$ is algebraically closed, then we simply have that $s(M_i,A)=(\dim_k V_i)/|G|$.

The purpose of this paper is to extend these results on the (generalized) $F$-signature and
the Frobenius limit to the case that $G$ is a finite group scheme, rather than a constant
finite group.
We say that the action of $G$ on $V$ is small if there exists some closed subset $F$ of
$V/\!/G=\Spec k[V]^G$ of codimension two or more such that $\pi: V\setminus \pi^{-1}(F)
\rightarrow V/\!/G\setminus F$ is a principal $G$-bundle ($G$-torsor), where
$\pi:V\rightarrow V/\!/G$ is the canonical map.
Note that the definition of the smallness is the natural generalization of the definition
for the constant group $G$ (that is, a faithful action without pseudo-reflection).
Our main theorem is the following.

\begin{Theorem}
  Let $k$ be a perfect field of characteristic $p>0$,
  $G$ be a finite $k$-group scheme over $k$,
  and $V$ a finite-dimensional $G$-module.
  Let $S=\Sym V$ be the symmetric algebra of $V$, and we assume that $S$ is graded so that
  each element of $V$ is homogeneous of degree one.
  Assume that the action of $G$ on $S$ is small.
  Let $M$ be a $\Bbb Q$-graded $S$-finite $S$-free $(G,S)$-module,
  and $L$ be its graded $(G,S)$-submodule which is reflexive as an $S$-module.
  Let $k=V_1,\ldots,V_r$ be the simple $G$-modules, and let $P_i$ be the projective cover of $V_i$.
  Then we have
\[
\FL(L^G)=\frac{\rank_S L}{\dim_k k[G]}[S'']=
\frac{\rank_S L}{\dim_k k[G]}
\sum_{i=1}^r \frac{\dim V_i}{\dim \End_G V_i} [(P_i\otimes_k S)^G],
\]
where $S''=(S\otimes_kk[G])^G$ is $S$ viewed as an $A$-module, where $A=S^G$.
\end{Theorem}

In particular, we have that $s(A)=1/\dim_kk[G]$ if $G$ is linearly reductive, and $s(A)=0$ otherwise.
This fact for the case that $G$ is linearly reductive
is deduced easily from \cite[Theorem~4.8]{Carvajal-Rojas}.
The case that $G$ is not linearly reductive ($s(A)=0$, or $A$ is not strongly $F$-regular)
is proved in \cite[Proposition~7.2]{LMM} for very small actions.

In section~2, we review some basic facts for small actions of group schemes.
Some of them are found in \cite{Hashimoto3} in very general forms which are much more than we need here,
and we gave shorter proofs for some of them here for convenience of readers.
In section~3, we prove our main theorem and some corollaries.

After the first version of this paper put on arXiv,
we got aware that \cite[Theorem~1.6]{LY} by Liedtke and Yasuda has some overlap with
Corollary~\ref{main-cor3.cor}.

The authors are grateful to Professor Anurag Singh, Professor Kei-ichi Watanabe, and
Professor Takehiko Yasuda for
valuable discussion and encouragement.

\section{Preliminaries}

\paragraph
Throughout this article, $k$ denotes a field.
For a $k$-scheme $X$, we denote by $k[X]$ the $k$-algebra $H^0(X,\O_X)$.

\paragraph
Let $G$ be an affine $k$-group scheme.
For an affine $G$-scheme $X$, $G$ acts on $k[X]$, the coordinate ring of $X$, by $(gf)(x)=f(g^{-1}x)$.
In other words, letting $\psi:k[X]\rightarrow k[G]\otimes k[X]$ be the map corresponding to the
action $G\times X\rightarrow X$, the comodule structure of $k[X]$ is the composite
\[
k[X]\xrightarrow \psi k[G]\otimes k[X]\xrightarrow{T} k[X]\otimes k[G]\xrightarrow{1_{k[X]}\otimes \Cal S} k[X]\otimes k[G],
\]
where $T(\alpha\otimes f)=f\otimes \alpha$, and $\Cal S$ is the antipode.
Note that $k[X]$ is a $k[G]$-comodule algebra ($k$-algebra $G$-module such that the product
$k[X]\otimes k[X]\rightarrow k[X]$ and the unit map $k\rightarrow k[X]$ are $G$-linear).

\paragraph
Let $G$ be a $k$-group scheme of finite type, and $N$ its normal closed subgroup scheme.
Let $f:X\rightarrow Y$ be a $k$-morphism between $k$-schemes of finite type on which $G$ acts.
We say that $f$ is a $G$-enriched principal $N$-bundle if $f$ is faithfully flat, $G$-equivariant, and
locally a trivial $N$-bundle with respect to the fppf topology.
As $f$ itself is fppf (i.e., faithfully flat and of finite presentation), this is
equivalent to say that the $G$-equivariant
$X$-morphism $\Psi: N\times X\rightarrow X\times_Y X$ given by
$\Psi(n,x)=(nx,x)$ is an isomorphism, where $G$ acts on $N$ by the conjugate action.

\begin{lemma}\label{Gro.lem}
  Assume that $G$, $N$, $X$, and $Y$ are all affine, and
  $f:X\rightarrow Y$ is a $G$-enriched principal $N$-bundle.
  Then $\Cal G:=(-)^N:\Mod(G,k[X])\rightarrow \Mod(G/N,k[Y])$ is an equivalence.
  The quasi-inverse is given by $\Cal F:=k[X]\otimes_{k[Y]}-$, and this is an
  equivalence of monoidal categories.
\end{lemma}

\begin{proof}
  It is easy to see that $\Cal G$ is right adjoint to $\Cal F$.
  Indeed,
\begin{multline*}
  \Hom_{G,k[X]}(k[X]\otimes_{k[Y]}L,M)\cong \Hom_{k[Y]}(L,M)^G\\
  =
  (\Hom_{k[Y]}(L,M)^N)^{G/N}=\Hom_{G/N,k[Y]}(L,M^N)
\end{multline*}
  in a natural way for $L\in\Mod(G/N,k[Y])$ and $M\in\Mod(G,k[X])$.

  To verify that the unit of adjunction $u:L\rightarrow (k[X]\otimes_{k[Y]}L)^N$ ($u(\alpha)=1\otimes\alpha$)
  and the counit of adjunction $\varepsilon: k[X]\otimes_{k[Y]} M^N\rightarrow M$ ($\varepsilon(f\otimes m)=fm$)
  are isomorphisms, we may and shall assume that $G=N$ and $G/N$ is trivial.
  
  As $k[Y]\rightarrow k[X]'$ is flat, where $k[X]'$ is the $k$ algebra $k[X]$
  with the trivial $N$-action, we have that $k[X]'\otimes_{k[Y]}(-)^N\rightarrow(k[X]'\otimes_{k[Y]}-)^N$
  is an isomorphism between the functors $\Mod(N,k[Y])\rightarrow \Mod k[X]'$.
  So taking the base change by $X'\rightarrow Y$, where $X'$ is the $k$-scheme $X$ with the
  trivial $N$-action, we may assume that $X=N\times Y$ is the trivial $N$-bundle.
  
  As the counit of adjunction $\varepsilon:\Cal F\Cal G\rightarrow \Id$ is given by
  $\Cal F\Cal G M=k[X]\otimes_{k[Y]}M^N\cong k[N]\otimes_k M^N\rightarrow M=\Id M$,
  where the last map is given by $f\otimes m\mapsto fm$ (note that
  $M$ is an $(N,k[N])$-module), which is an isomorphism, see \cite[Theorem~4.1.1]{Sweedler}.
  Conversely, the unit of adjunction
  $L\rightarrow \Cal G\Cal F L=(k[X]\otimes_{k[Y]}L)^N\cong (k[N]\otimes_kL)^N$ given by
  $n\mapsto 1\otimes n$ is an isomorphism for $L\in \Mod k[Y]$.
  Thus $\Cal F$ and $\Cal G$ are quasi-inverse each other.

  As $\Cal F$ preserves the monoidal structure, the equivalence is that of monoidal categories.
\end{proof}

\paragraph
Let $G$ be a $k$-group scheme of finite type, $N$ its normal closed subgroup scheme,
and $f:X\rightarrow Y$ a $G$-enriched almost principal $N$-bundle if
$f$ is $G$-equivariant, the action of $N$ on $Y$ is trivial, there exist $G$-stable
open subset $V$ of $Y$ and $U$ of $f^{-1}(V)$ such that $\codim(Y\setminus V,Y)\geq 2$,
$\codim(X\setminus U,X)\geq 2$, and $f_U:U\rightarrow V$ is a $G$-enriched principal $N$-bundle.
Considering the case that $G=N$,
a $G$-enriched almost principal $G$-bundle is simply called an almost principal $G$-bundle.
For basics on almost principal $G$-bundles, see \cite{Hashimoto3}.

\begin{lemma}\label{almost-Grothendieck-theorem.lem}
  Let $N$ be a normal closed subgroup scheme of $G$.
  Assume that $f:X\rightarrow Y$ is a $G$-enriched almost principal $N$-bundle with $G$, $X$ and $Y$ are all affine.
  Assume that both $X$ and $Y$ are normal.
  Then
  \begin{enumerate}
  \item[\rm(1)] The canonical map $k[Y]\rightarrow k[X]^N$ is an isomorphism of $G/N$-algebras.
  \item[\rm(2)] The functors $\Cal G=(-)^N:\Ref(G,k[X])\rightarrow \Ref(G/N,k[Y])$ and
    $\Cal F=(k[X]\otimes_{k[Y]}-)^{**}:\Ref(G/N,k[Y])\rightarrow \Ref(G,k[X])$ are
    quasi-inverse each other, and give an equivalence of monoidal categories,
    where $\Ref(G/N,k[Y])$ denotes the category of $(G/N,k[Y])$-modules
    which are finitely generated reflexive as $k[Y]$-modules,
    and $\Ref(G,k[X])$ denotes the category of $(G,k[X])$-modules
    which are finitely generated reflexive as $k[X]$-modules.
    The counit $\varepsilon: (k[X]\otimes_{k[Y]}M^G)^{**}\rightarrow M$
    is the double dual of $\alpha\otimes m\mapsto \alpha m$, and the unit
    $u: L\rightarrow ((k[X]\otimes_{k[Y]}-)^{**})^N$ is the composite
    \[
    L\cong k[Y]\otimes_{k[Y]}L\rightarrow (k[X]\otimes_{k[Y]}L)^N\rightarrow
    ((k[X]\otimes_{k[Y]}L)^{**})^N.
    \]
    \item[\rm(3)] The equivalence preserves the rank of modules.
  \end{enumerate}
\end{lemma}

\begin{proof}
  We may assume that $G=N$.
  We may discuss componentwise, and we may assume that $Y$ is connected, and hence $k[Y]$ is
  a normal domain.
  There is an open subset $V$ of $Y$ and an $N$-stable open subset $U$ of $f^{-1}(V)$ such that
  $f:U\rightarrow V$ is a principal $N$-bundle, $\codim(Y\setminus V,Y)\geq 2$, and $\codim (X\setminus U,X)\geq 2$.

  First consider the case that $V=Y$.
  As $G=N$ is affine, $f:U\rightarrow Y$ is affine, and hence $U$ is affine.
  As $X$ is normal and $\codim(X\setminus U,X)\geq 2$, we have that $U=X$.
  Thus $f$ itself is a principal $N$-bundle.
  Then $M\in\Mod(N,k[X])$ is finitely generated as a $k[X]$-module and is reflexive if and only if
  $M^N\in\Mod(N,k[Y])$ is finitely generated as a $k[Y]$-module and is reflexive.
  Indeed, As $k[Y]\rightarrow k[X]$ is faithfully flat and $M\cong k[X]\otimes_{k[Y]}M^N$, we have that
  $M$ is finitely generated if and only if $M^N$ is finitely generated.
  If this is the case,
  \[
  (M^*)^N=\Mod(N,k[X])(M,k[X])\cong \Mod(k[Y])(M^N,k[X]^N)=(M^N)^*.
  \]
  In particular, $M$ is reflexive if and only if $M^N$ is so.
  So the assertion of the lemma follows from Lemma~\ref{Gro.lem} this case.
  In particular, the lemma is true if $\dim Y\leq 1$.
  So the lemma is true for $f_P:X_P\rightarrow Y_P$ for every $P\in \Spec k[Y]$ with
  $\height P\leq 1$, where $f_P:X_P=\Spec k[X]_P\rightarrow Y_P=\Spec k[Y]_P$ is the
  base change of $f$.
  
  As $k[X]$ is normal, we can write $k[X]=B_1\times\cdots\times B_r$, where $B_i$ is a normal domain.
Let $M\in\Ref(G,k[X])$, and $Q$ be a height-one prime ideal of $k[X]$.
  Then $Q$ as a point of $X$ lies in $U$.
  As $f:U\rightarrow Y$ is flat, we have that $\height(Q\cap k[Y])\leq 1$.
  Let $M_i=B_i\otimes_{k[X]}M$.
  Then
  \begin{multline*}
  M_i
  \subset
  \bigcap_{P\in\Spec k[Y],\;\height P\leq 1} (M_i)_P\\
  =
  \bigcap_{P\in\Spec k[Y],\;\height P\leq 1}\bigcap_{Q\in\Spec B_i,\;\height Q\leq 1,\;Q\cap k[Y]\subset P} (M_i)_Q
  =
  \bigcap_{Q\in \Spec B_i,\;\height Q\leq 1}(M_i)_Q=M_i,
  \end{multline*}
  where the intersection is taken in $Q(B_i)\otimes_{B_i}M_i$.
  This shows that
  \begin{multline*}
  M=\prod_{i=1}^r M_i=\prod_{i=1}^r(\bigcap_{P\in\Spec k[Y],\;\height P\leq 1}(M_i)_P)\\
  =\bigcap_{P\in\Spec k[Y],\;\height P\leq 1}\prod_{i=1}^r(M_i)_P
  =\bigcap_{P\in\Spec k[Y],\;\height P\leq 1}M_P.
  \end{multline*}
  So
  \[
  M^N=(\bigcap_PM_P)^N=\bigcap_P (M_P)^N=\bigcap_P(M^N)_P.
  \]
  As $Y$ is a Noetherian space and hence the subset $\{P\in\Spec k[Y]\mid \height P=1\}$ of $\Spec k[Y]$ is quasi-compact,
  we have that there is a finitely generated $k[Y]$-submodule $L$ of $M^N$ such that $L_P=(M^N)_P$ for
  each $P$.
  So $M^N=\bigcap_P(M^N)_P=\bigcap_P L_P=L^{**}$ is also finitely generated.
  As $(M^N)_P\cong (M_P)^N$ is reflexive and hence is free, $M^N\cong \bigcap_P(M^N)_P$ is also reflexive.
  Thus $\Cal G$ is well-defined.
  Note also that $\Cal F$ is also well-defined, since a double dual of a finitely generated
  module over a Noetherian ring is a second syzygy.
  
We want to prove that $\Cal F L\xrightarrow{u}\Cal F\Cal G\Cal F L\xrightarrow\varepsilon \Cal F L$ 
  is the identity, and $\Cal G M \xrightarrow u \Cal G\Cal F\Cal G M\xrightarrow\varepsilon\Cal G M$
  is the identity.
  As the canonical map
  \[
  \Hom_{k[Y]}(\Cal G M,\Cal G M)\rightarrow\prod_P \Hom_{k[Y]_P}(\Cal G_P M_P,\Cal G_P M_P)
\]
  is injective,
  where the product runs through all the minimal prime ideals $P$ of $k[Y]$, and
  $\Cal G_P:\Ref(G,k[X]_P)\rightarrow \Ref(k[Y]_P)$ is the functor $(-)^G$,
  to see that the element $\id_{\Cal G M}-\varepsilon u$ is zero, we may assume that $k[Y]$ is a field, and
  this case is trivial, where $\varepsilon u$ is the composite
  \[
  \Cal G M\xrightarrow u \Cal G\Cal F\Cal G M\xrightarrow\varepsilon \Cal G M.
  \]
  Using a similar argument, we can prove that $\id_{\Cal F N}-\varepsilon u$ is zero,
  and thus $\varepsilon u=\id$.
  So $\Cal G$ is right adjoint to $\Cal F$.
  Moreover, since $\Cal F\Cal G M$ is reflexive if $M$ is reflexive, the counit $\varepsilon :\Cal F\Cal G M
  \rightarrow M$ is an isomorphism if and only if $\varepsilon_P:\Cal F_P\Cal G_P M_P\rightarrow M_P$ is
  an isomorphism for $P\in\Spec k[Y]$ whose height is at most one.
  As $f_P:X_P\rightarrow Y_P$ is a principal $G$-bundle and a reflexive $k[X]_P$-module and a reflexive
  $k[Y]_P$-module are free over $k[X]_P$ and $k[Y]_P$, respectively for $P\in\Spec k[Y]$ whose height is
  at most one, we have that the counit $\varepsilon:\Cal F\Cal G \rightarrow\Id_{\Ref(G,k[X])}$ is an
  isomorphism.
  Similarly, the unit of adjunction $u:\Id_{\Ref(k[Y])}\rightarrow\Cal G\Cal F$ is also an isomorphism.
  Thus $\Cal F$ and $\Cal G$ are quasi-inverse each other.

  As the unit of adjunction $u: k[Y]\rightarrow \Cal G\Cal F k[Y]=k[X]^N$ is an isomorphism,
  the first assertion is now clear.

  Let $L,L'\in\Ref(G/N,k[Y])$, and consider the canonical map
  \[
  (k[X]\otimes_{k[Y]}(L\otimes_{k[Y]}L'))^{**}
  \rightarrow
  (k[X]\otimes_{k[Y]}(L\otimes_{k[Y]}L')^{**})^{**}.
  \]
  This is a morphism in $\Ref(G,k[X])$.
  To see that this is an isomorphism, we may localize at $P\in\Spec k[Y]$ with $\height P=1$,
  and this case is obvious.
  Similarly, to see that 
  \[
  ((k[X]\otimes_{k[Y]}L)\otimes_{k[X]}(k[X]\otimes_{k[Y]}L'))^{**}
  \rightarrow
  ((k[X]\otimes_{k[Y]}L)^{**}\otimes_{k[X]}(k[X]\otimes_{k[Y]}L')^{**})^{**}
  \]
  is an isomorphism, it suffices to show that this is bijective, and
  we may localize at a height one prime $P$ of $k[Y]$ again, and this case is obvious.
  
  Thus $\Cal F$ preserves the monoidal structure and the rank of modules.
\end{proof}
  
\paragraph
From now on, we assume that $k$ is a perfect field of characteristic $p>0$, and
$G$ be a finite $k$-group scheme.
  We set $N=G^\circ$, the identity component of $G$, and $H=G\red$.
  As $k$ is perfect, $H$ is a closed subgroup scheme of $G$, and is \'etale over $k$.
  Note that the composite $H\hookrightarrow G\rightarrow G/N\cong \pi_0G$ is an isomorphism,
  where $\pi_0G$ is the unique maximal \'etale quotient of $G$.
  In other words, $k[\pi_0G]$ is the $k$-subalgebra of $k[G]$ generated by all the
  \'etale $k$-subalgebras of $k[G]$.
  So $G$ is the semidirect product $G=N\rtimes H$.

\begin{lemma}\label{infinitesimal-invariant-finite.lem}
    There exists some $e_0\geq 1$ such that $B^{p^{e_0}}\subset B^N\subset B$,
    where $B^{p^{e_0}}$ is the image of the $e$th Frobenius map $F^e:B\rightarrow B$.
    In particular, $B^N\rightarrow B$ is finite, and $B^N$ is finitely generated
    over $k$.
\end{lemma}

\begin{proof}
  Let $J=\Ker \varepsilon$ be the kernel of the counit map $k[N]\rightarrow k$.
  Then $k[N]=k\oplus J$ as a $k$-vector space.
  As $N$ is infinitesimal (that is, $N\red=\Spec k$), $J$ is a nilpotent ideal, and
  hence there exists some $e_0\geq 1$ such that $J^{[p^{e_0}]}=0$, where
  $J^{[p^{e_0}]}$ is the ideal of $k[N]$ generated by $\{a^{p^{e_0}}\mid a\in J\}$.
  Then for $c+a\in k\oplus J=k[N]$, we have that
  $F^{e_0}(c+a)=c^{p^{e_0}}=F^{e_0}(u\varepsilon(c+a))$, where $u:k\rightarrow k[N]$ is the
  unit map.
  Hence $F^{e_0}=F^{e_0}u\varepsilon$.
  So for $b\in B$,
  \[
  \omega_B(b^{p^{e_0}})=(\omega_Bb)^{p^{e_0}}=\sum_{(b)}b_{(0)}^{p^{e_0}}\otimes b_{(1)}^{p^{e_0}}
  =\sum_{(b)}(b_{(0)}\varepsilon(b_{(1)}))^{p^{e_0}}\otimes 1=b^{p^{e_0}}\otimes 1.
  \]
  That is, $b^{p^{e_0}}\in B^N$, and the first assertion has been proved.
  
  As $B$ is finitely generated over a perfect field, $B$ is $F$-finite.
  So $B$ is finite over $B^{p^{e_0}}$.
  Hence, both $B^{p^{e_0}}\rightarrow B^N$ and $B^N\rightarrow B$ are also finite,
  and we are done.
\end{proof}

\begin{lemma}\label{invariant-finite.lem}
  $B^G$ is a finitely generated $k$-algebra, and $B^G\rightarrow B$ is finite.
\end{lemma}

\begin{proof}
  As $B^G=(B^N)^H$, we may assume that either $G=N$ or $G=H$.
  The case that $G=N$ is done in Lemma~\ref{infinitesimal-invariant-finite.lem}.
  The case that $G=H$ is reduced to the case that $k$ is algebraically closed.
  In that case, $G$ is a constant finite group, and this is well-known.
\end{proof}

\paragraph\label{W'.par}
For an $H$-module $W$, we denote its restriction by the canonical homomorphism $G\rightarrow G/N\cong H$ by $W'$.
Thus the restriction of $W'$ on $N$ is trivial, while the restriction of $W'$ on $H$ is $W$.
We adopt this notation for $G$-modules $M$.
We regard $M$ as its restriction to $H$, and then consider the $G$-module $M'$.
Thus $M'$ is the restriction of the $G$-module $M$ to $G$ with respect to the homomorphism
$\rho:G\rightarrow G/N\cong H\hookrightarrow G$.

\paragraph
The group scheme $G$ viewed as the $G$-scheme with the left (resp.\ right) regular $G$-action is denoted by $G_l$ (resp.\ $G_r$).
That is, the action $G\times G_l\rightarrow G_l$ is given by $(g,g_1)\mapsto gg_1$
(resp.\ $G\times G_r\rightarrow G_r$ is given by $(g,g_2)\mapsto g_2g^{-1}$).
Note that the inverse $\iota:G_l\rightarrow G_r$ is the isomorphism of $G$-schemes.
Note also that the coaction of $k[G_l]$ is given by $f\mapsto \sum_{(f)}f_{(2)}\otimes \Cal Sf_{(1)}$, where
$\Cal S:k[G]\rightarrow k[G]$ is the antipode.
The coaction of $k[G_r]$ is the coproduct $f\mapsto \sum_{(f)}f_{(1)}\otimes f_{(2)}$.

\paragraph
We consider that $G=N\rtimes H$ acts on $N_l$ by $(nh)(n_1)=nhn_1h^{-1}$, and on $H_l$ by $(nh)(h_1)=hh_1$.
Note that the product $N_l\times H_l$ is isomorphic to $G_l$ by $(n_1,h_1)\mapsto n_1h_1$.
Similarly, $H_r\times N_r\rightarrow G_r$ given by $(h,n)\mapsto hn$ is an isomorphism,
where $H$ acts on $N_r$ by $(h,n)\mapsto hnh^{-1}$, and $N$
acts on $H_r$ trivially.
In particular, we get an isomorphism of $G$-modules $k[G_r]\cong k[H_r]'\otimes k[N_r]$.

\paragraph
A $G$-module $W$ is both an $N$-module and an $H$-module, and the composite of the actions $h\circ n\circ h^{-1}$
agrees with the action of $hnh^{-1}\in N$.
The converse is also true, and a $G$-linear map is nothing but an $N$-linear $H$-linear mapping.

\begin{lemma}
  For a $G$-module $W$, the map
  $\Box: W\otimes k[N_r] \rightarrow W'\otimes k[N_r]$ given by $w\otimes \alpha\mapsto
  \sum_{(w)}w_{(0)}\otimes w_{(1)}\alpha$ is an isomorphism of $(G,k[N_r])$-modules.
\end{lemma}

\begin{proof}
As
\begin{multline*}
\beta\Box(w\otimes \alpha)=\beta(\sum_{(w)}w_{(0)}\otimes w_{(1)}\alpha)
=\sum_{(\beta)}\sum_{(w)}\beta_{(1)}w_{(0)}\otimes \beta_{(2)}w_{(1)}\alpha\\
=\sum_{(\beta w)}(\beta w)_{(0)}\otimes(\beta w)_{(1)}\alpha
=\Box(\beta(w\otimes\alpha)),
\end{multline*}
we have that $\Box$ is $k[N_r]$-linear.
It is easy to see that $w'\otimes\beta\mapsto \sum_{(w')}w'_{(0)}\otimes (\Cal S w'_{(1)})\beta$ is the inverse
of $\Box$, and $\Box$ is bijective.

It remains to show that the map $\Box$ is $H$-linear and $N$-linear.
As
\begin{multline*}
(\Box\otimes 1_{k[H]})\omega(w\otimes\alpha)=
(\Box\otimes 1_{k[H]})\left(\sum_{(w),(\alpha)}w_{(0)}\otimes \alpha_{(2)}\otimes w_{(1)}(\Cal S\alpha_{(1)})\alpha_{(3)}\right)\\
=
\sum_{(w),(\alpha)}w_{(0)}\otimes w_{(1)}\alpha_{(2)}\otimes w_{(2)}(\Cal S\alpha_{(1)})\alpha_{(3)}
\end{multline*}
and
\begin{multline*}
  \omega\Box(w\otimes\alpha)=\omega(\sum_{(w)}w_{(0)}\otimes w_{(1)}\alpha)
\\
  =\sum_{(w),(\alpha)}w_{(0)}\otimes w_{(3)}\alpha_{(2)}
  \otimes w_{(1)}(\Cal S w_{(2)})w_{(3)}(\Cal S\alpha_{(1)})\alpha_{(3)}
  \\
=\sum_{(w),(\alpha)}w_{(0)}\otimes w_{(1)}\alpha_{(2)}\otimes w_{(2)}(\Cal S\alpha_{(1)})\alpha_{(3)},
\end{multline*}
we have that $\Box$ is $H$-linear.
The fact that $\Box$ is $N$-linear is checked more easily, and thus $\Box$ is $G$-linear.
\end{proof}

\paragraph
Let $W$ be a $k$-vector space.
For $e\in\ZZ$, we define ${}^eW$ to be the additive group $W$ with a new $k$ action given by $\alpha\cdot w=\alpha^{p^e}w$ for
$\alpha\in k$ and $w\in W$.
For $w\in W$, we denote the element $w\in W={}^eW$ by ${}^ew$.
It is easy to see that ${}^e(-):\Mod k\rightarrow \Mod k$ is an auto-equivalence of monoidal categories.
For a $k$-algebra $B$, ${}^eB$ is a $k$-algebra, and the Frobenius map $F^e:B\rightarrow {}^eB$ is a $k$-algebra map.
If $X=\Spec B$, then we denote ${}^eX=\Spec {}^eB$.
For a $B$-module $M$, ${}^eM$ is a ${}^eB$-module in a natural way, and hence is a $B$-module.

\paragraph  
Let $f: R\rightarrow R'$ be a $k$-algebra map between $k$-algebras of finite type.
Then $f$ is \'etale if and only if ${}^eR\otimes_R R'\rightarrow {}^eR'$ given by ${}^e r\otimes x\mapsto{}^e(rx^{p^e})$
is an isomorphism \cite{Radu}.
In particular, since $H$ is \'etale over $k$, the Frobenius map $F^e:{}^eH\rightarrow H$ is an isomorphism of groups.
If $W$ is an $H$-module, then we have a homomorphism of $k$-groups $\psi:H\rightarrow \GL(W)$.
Note that ${}^e\GL(W)$ acts on ${}^eW$ via the action ${}^e\GL(W)\times {}^eW\cong{}^e(\GL(W)\times W)\rightarrow {}^eW$.
The coaction of ${}^eW$ is
\[
  {}^eW\xrightarrow{{}^e\omega}{}^e(W\otimes k[H])\cong {}^eW\otimes {}^ek[H]\xrightarrow{1\otimes F^{-e}}{}^eW\otimes k[H].
\]
In particular, if $H$ is a constant finite group, $h\in H$, $w_1,\ldots,w_n$ a $k$-basis of $W$, and $hw_j=\sum_i c_{ij}w_i$
($c_{ij}\in k$), then ${}^ew_1,\ldots,{}^ew_n$ is a $k$-basis of ${}^eW$, and $h{}^ew_j=\sum_i c_{ij}^{1/p^e}{}^ew_i$.

Although ${}^eW$ is an $H$-module again for any $H$-module $W$, it seems that there is no canonical way to make
${}^eM$ a $G$-module for a $G$-module $M$.
  
\section{$F$-signatures of the rings of invariants}

\paragraph
Let $k$ be a perfect field of characteristic $p>0$, $A=\bigoplus_{i\geq 0}A_i$ be a finitely generated commutative graded
$k$-algebra such that $A_0$ is finite over $k$.
Let $T=\bigoplus_{i\geq 0}T_i$ be a finite graded $A$-algebra which might not be commutative.
We define $\Cal C(T)$ the category of $\Bbb Q$-graded finitely generated left $T$-modules.
Let $\Theta^*(T)$ denote $(\bigoplus_{M\in\Ob(\Cal C(T))}\Bbb R\cdot[M])/([L]-[M]-[N])$,
the $\Bbb R$-vector space with the set of objects $[M]$
of $\Cal C(T)$ its basis modulo the relations $[L]-[M]-[N]$ for objects $L,M,N\in\Cal C(T)$
such that $L\cong M\oplus N$.
As the endomorphism ring of any object of $\Cal C(T)$ is a finite dimensional algebra, 
$\Theta^*(T)$ has $\Ind(\Cal C(T))$, the set of isomorphism classes of indecomposable objects of $\Cal C(T)$ as
its basis by the Krull--Schmidt theorem.
We set $\Theta^\circ(T)=\Theta^*(T)/([M]-[M(c)]\mid M\in\Cal C(T),\; c\in\Bbb Q)$.
Note that $\Theta^\circ(T)$ has $\Ind^\circ(\Cal C(T))=\Ind(\Cal C(T))/\mathord\sim$ as its basis,
where $\mathord\sim$ is the equivalence
relation of $\Ind(\Cal C(T))$ such that $[M]\sim[N]$ if and only if there exists some $c\in\QQ$ such that
$N\cong M(c)$.

\paragraph\label{normed.par}
For $\alpha=\sum_{M\in\Ind^\circ(\Cal C(T))}c_M[M]\in\Theta^\circ(T)$ ($c_M\in\RR$), we define the norm
$\norm\alpha$ of $\alpha$ by $\norm\alpha:=\sum_M|c_M|u_T(M)$, where $u_T(M)=\ell_T(M/JM)$, where
$J$ is the graded radical of $T$ (note that $T/J$ is a finite-dimensional algebra).
It is easy to see that $(\Theta^*(T),\norm-)$ is a normed space.

\paragraph
Let $B=\bigoplus_{i\geq 0}B_i$, $B_0=k$, be a
finitely generated positively graded $k$-algebra.
Let $G$ be a finite $k$-group scheme.
We denote by $|G|$ the dimension $\dim k[G]$ of $k[G]$.
Let $G$ act on $B$.
Assume that the action of $G$ on $B$ is degree-preserving.
That is, each $B_i$ is a $G$-submodule of $B$ for any $i$.
As $B$ is a module algebra over the dual Hopf algebra $k[G]^*$ of $k[G]$,
we can define the crossed product (smash product) $T=k[G]^*\#B$, see \cite[Chapter~4]{Montgomery}.
By Lemma~\ref{invariant-finite.lem}, $T$ is a finite algebra over $A:=B^G$.
We denote $\Cal C(T)$ and $\Theta^\circ(T)$ by $\Cal C(G,B)$ and $\Theta^\circ(G,B)$, respectively.
Thus an element $\alpha$ of $\Cal C(G,B)$ can be written as $\alpha=\sum_M c_M[M]$ uniquely, where
$M$ runs through $\Ind^\circ(\Cal(G,B))$.
We say that $\alpha\geq 0$ if $c_M\geq 0$ for each $M$.
We define $\Cal C^+(G,B)=\{\alpha\in\Cal C(G,B)\mid \alpha\geq 0\}$.

\paragraph
  For $M,L\in\Cal C(G,B)$, we define $\summ(M,L)$ to be the supremum of $r\in\ZZ_{\geq 0}$ such that
  $M^r$ is a direct summand of $L$.
  More generally, for $M\in\Cal C(G,B)$ and $\alpha\in\Cal C^+(G,B)$, we define $\summ(M,\alpha)$ to be
  the supremum of $r\in\ZZ_{\geq 0}$ such that $\alpha-r[M]\geq 0$.
  If $M$ is indecomposable and $\alpha=\sum_{N}c_N[N]\in\Cal C^+(G,B)$,
  then $\summ(M,\alpha):=c_M$.

  \paragraph
  Let $M,L\in\Cal C(G,B)$, and $f:M\rightarrow L$ be a $G$-linear $B$-linear mapping.
  We say that $f$ is gradable if there exists some direct decomposition $M=\bigoplus_{i=1}^r M_i$ in
  $\Cal C(G,B)$ and $c_1,\ldots,c_r\in\QQ$ such that $f:\bigoplus_{i=1}^r M_i(c_i)\rightarrow L$ is
  a morphism in $\Cal C(G,B)$ (that is, degree-preserving).

\paragraph
Let $H$ be an \'etale $k$-group scheme.
Let $\Cal C=\Cal C(H,B)$, where $B=\bigoplus_{i\geq 0}B_i$ with $B_0=k$ is a finitely generated
positively graded $k$-algebra with a degree-preserving $H$-action.
  It is easy to see that if $M\in\Cal C$, then ${}^eM\in\Cal C$, and ${}^e(-):\Cal C\rightarrow\Cal C$
  is an exact functor.
  Note also that ${}^e(M(c))\cong {}^eM(c/p^{e})$.
  So we have a well-defined
  $\Bbb R$-linear map ${}^e(-):\Theta^\circ(H,B)\rightarrow \Theta^\circ(H,B)$ given by ${}^e[M]=[{}^eM]$.

\paragraph
  Let $d=\dim B$.
  We define the Frobenius limit \cite{HS} of $\alpha\in \Theta^\circ(H,B)$ by
  $\FL \alpha=\lim_{e\rightarrow\infty}{}^e\alpha/p^{de}$, if the limit in the right-hand
  side exists.

  We define the $H$-equivariant generalized $F$-signature, denoted by
  $s_H(M,\alpha)$, to be $\lim_{e\rightarrow\infty}\summ(M,{}^e\alpha)/p^{de}$ if the limit exists.
  If the Frobenius limit $\FL(\alpha)$ exists, 
  then $s_H(M,\alpha)=\summ(M,\FL(\alpha))$.
  If $H=\{e\}$ is trivial, then we simply denote $s_{\{e\}}(M,B)$ by $s(M,B)$.

\paragraph
Let $k$ be a perfect field of characteristic $p>0$, and $G$ be a finite $k$-group scheme.

\begin{lemma}
  Assume that the action of $G$ on $B$ is small, and $B$ is a normal domain.
  \begin{enumerate}
  \item[(1)] $M\mapsto M^G$ from the category of $\Bbb Q$-graded finitely generated
    reflexive $(G,B)$-modules $\Cal R(G,B)$ to the category of $\Bbb Q$-graded finitely generated
    reflexive $B^G$-modules $\Cal R(B^G)$ is an equivalence, whose quasi-inverse is given by
    $L\mapsto (L\otimes_{B^G}B)^{**}$.
  \item[(2)] Assume that $G$ is \'etale.
    Then for an indecomposable object $M\in\Cal R(G,B)$, we have that $s_G(M,B)=s(M^G,B^G)$,
    provided $s_G(M,B)$ exists.
  \end{enumerate}
\end{lemma}

\begin{proof}
  (1).
  This is an obvious extension of Lemma~\ref{almost-Grothendieck-theorem.lem}.
  
  (2).
  Note that the $G$-invariance $(-)^G$ induces an $\Bbb R$-linear map $[(-)^G]:\Theta^\circ(G,B)
  \rightarrow \Theta^\circ(B^G)$.
  By (1), it is an isomorphism from the subspace $\Theta_{\Ref}^\circ(G,B)$
  of $\Theta^\circ(G,B)$ generated by the reflexive
  objects to the subspace $\Theta_{\Ref}^\circ(B^G)$ of $\Theta^\circ(B^G)$
  generated by the reflexive objects.
  Note that $({}^eB)^G\cong {}^e(B^G)$.
  As $[(-)^G]$ is continuous,
  \[
  \FL(B^G)=\lim_{e\rightarrow\infty}{}^e[B^G]/p^{de}=(\lim_{e\rightarrow \infty}{}^e[B]/p^{de})^G
  =\FL(B)^G,
  \]
  where $d=\dim B$.
  We can write $\FL(B)=\sum_{L\in\ind^\circ(B^G)}c_L[L]$ ($c_L\in\RR$).
  Then $\FL(B^G)=\FL(B)^G=\sum_L c_L[L^G]$.
  By (1), we have that $s(M^G,B^G)=c_M=s_G(M,B)$.
\end{proof}

\paragraph
Let $G$ be a finite group scheme over $k$, and $V$ a finite-dimensional $G$-module.
Let $S=\Sym V^*$ be the symmetric algebra of $V$, and we identify $V$ with $\Spec S$.
We have that $G=G^\circ \rtimes G\red$.
By \cite[(10.8)]{Hashimoto2}, the action of $G$ on $V$ is small if and only if the action of
$G^\circ$ on $V$ is small and the action of $G\red$ on $V\aq G^\circ=\Spec S^{G^\circ}$ is small.

In what follows, we assume that the action of $G$ is small.
In particular, $G\rightarrow\GL(V)$ is a closed immersion.

\paragraph
  Let $k$ be a perfect field of characteristic $p>0$,
  $G$ be an \'etale $k$-group scheme over $k$,
  and $V$ a finite-dimensional $G$-module.
  Let $S=\Sym V$ be the symmetric algebra of $V$, and we assume that $S$ is graded so that
  each element of $V$ is homogeneous of degree one.
  Assume that the action of $G$ on $S$ is small.
  Let $\tilde G=G\times\Bbb G_m$.
  Let $X=\Spec S$.
  Then $\tilde G$ acts on $X$.
  Let $U$ be the \'etale locus of the quotient map $\pi:X\rightarrow Y=\Spec S^G$.
  Then $U$ is a dense open subset of $X$.

\begin{lemma}\label{split-mono.lem}
  There is a split monomorphism
  $W\hookrightarrow S$ of $\tilde G$-modules,
  where $W$ is a $\tilde G$-module whose underlying
  $G$-module is isomorphic to $k[G]$.
\end{lemma}

\begin{proof}
  Let $g=\dim k[G]$.
  If there is no $k$-rational point of $U$, then since $U$ is a dense open subset
  of an affine space, $k=\Bbb F_q$ is a finite field.
  In this case, we can take a prime number $\ell>g$
  such that $U$ has a $k_1:=\Bbb F_{q^\ell}$-valued point $x$.
  Then the image $y$ of $x$ by $\pi:X\rightarrow Y$ is again a $k_1$-valued point.
  As $G\times U\rightarrow U\times_Y U$ is an isomorphism, we have that $G\times\{x\}\rightarrow Gx$
  is an isomorphism.
  Or equivalently, $S\xrightarrow \omega S\otimes k[G]\rightarrow k_1[G]$ is surjective.
  As $k_1[G]$ as a $G$-module is nothing but $k[G]^\ell$, there is a surjective $G$-linear map
  $\rho:S\rightarrow k[G]$.
  As $k[G]^*$ is a finite dimensional Hopf algebra, it is a Frobenius algebra \cite[Theorem~3.6]{SY}.
  So the injective $G$-module ($kG=k[G]^*$-module)
  $k[G]$ is also projective.
  So $\rho$ splits, and $k[G]$ is a summand of $S$ as a $G$-module.
  So there exists some $r\geq 0$ such that $k[G]$ is a summand of $\bigoplus_{i=0}^r S_i$,
  where $S_r$ denotes the homogeneous component of $S$ of degree $r$.
  By the Krull--Schmidt theorem, $\bigoplus_{i=0}^r S_i$ has a graded summand which is isomorphic
  to $k[G]$ as a $G$-module, and this is what we wanted to prove.
\end{proof}

\paragraph
Just mimicking the proof in \cite[section~4]{HS},
we can extend \cite[Theorem~4.13]{HS} to the actions
of \'etale group schemes.
Lemma~\ref{split-mono.lem} above can be used to extend \cite[Lemma~4.11]{HS} to the
case that $G$ is a general \'etale group scheme.

\begin{proposition}\label{etale-case.prop}
  Let $k$ be a perfect field of characteristic $p>0$,
  $G$ be an \'etale $k$-group scheme, and $V$ be a finite-dimensional $G$-module.
  Let $S=\Sym V$ be the symmetric algebra of $V$, and we assume that $S$ is graded so that
  each element of $V$ is homogeneous of degree one.
  Let $M$ be a $\Bbb Q$-graded $S$-finite $S$-free $(G,S)$-module.
  Let $k=V_1,\ldots,V_r$ be the simple $G$-modules, and let $P_i$ be the
  projective cover of $V_i$.
  Then we have
\[
\FL(M)=\frac{\rank_S M}{\dim_k k[G]}[k[G]\otimes_k S].
\]
In particular, we have
\[
\FL(M^G)=\frac{\rank_S M}{\dim_k k[G]}[S]=
\frac{\rank_S M}{\dim_k k[G]}
\sum_{i=1}^r \frac{\dim V_i}{\dim \End_G V_i} [(P_i\otimes_k S)^G].
\]
\end{proposition}

Using Proposition~\ref{etale-case.prop} above, we can prove the following.

\begin{proposition}\label{etale-case2.prop}
  Let the notation be as in Proposition~\ref{etale-case.prop}.
  Let $L$ be a graded $(G,S)$-submodule of $M$, and assume that $L$ is reflexive as an $S$-module.
  Then
  \[
  \FL(L)=\frac{\rank_S L}{\dim_k k[G]}[k[G]\otimes_k S].
  \]
  In particular, 
we have
\[
\FL(L^G)=\frac{\rank_S L}{\dim_k k[G]}[S]=
\frac{\rank_S L}{\dim_k k[G]}
\sum_{i=1}^r \frac{\dim V_i}{\dim \End_G V_i} [(P_i\otimes_k S)^G].
\]
\end{proposition}

\begin{proof}
  We set $d=\dim_k V$, $g=\dim_kk[G]$, $m=\rank_SM$ and $\ell=\rank_S L$.
  For $e\geq 1$, we set $\gamma_e:=\summ(k[G]\otimes_k S,{}^eM)$.
  By Proposition~\ref{etale-case.prop}, we have that $\lim_{e\rightarrow \infty}\gamma_e/p^{de}=m/g$.

  Let $A=S^G$.
  Then there is a gradable injective $A$-linear mapping $A^\ell\rightarrow L^G$, which induces a
  gradable injective $(G,S)$-linear mapping $h: S^\ell\rightarrow L$.
  For each $e\geq 1$, consider the composite
  \[
    \delta_e: {}^e S^\ell \xrightarrow {{}^eh} {}^eL \hookrightarrow {}^eM\rightarrow (k[G]\otimes_k S)^{\gamma_e}
    \rightarrow (k[G]\otimes_k S)^{\gamma_e}/\fm(k[G]\otimes_k S)^{\gamma_e}\cong k[G]^{\gamma_e},
  \]
  where $\fm=S_+$ is the irrelevant ideal.
  Let $W_e:=\Image\delta_e$, and we write $W_e=k[G]^{\mu_e}\oplus U_e$ such that $U_e$ does not have $k[G]$
  as a direct summand.
  Then since $k[G]$ is a projective $G$-module,
  the surjective map ${}^e S^\ell\xrightarrow{\delta_e} W_e\rightarrow k[G]^{\mu_e}$ has a gradable $G$-linear splitting
  $\psi_e: k[G]^{\mu_e}\rightarrow {}^eS^\ell$.
  This induces a unique gradable $(G,S)$-linear map $\tilde\psi_e: (k[G]\otimes S)^{\mu_e}\rightarrow {}^eS^\ell$,
  and
  \[
  \lambda_e: (k[G]\otimes_k S)^{\mu_e}\xrightarrow{\tilde\psi_e}{}^eS^\ell \rightarrow{}^eL\rightarrow  {}^eM\rightarrow (k[G]\otimes_k S)^{\gamma_e}
  \]
  induces a split monomorphism $\bar\lambda_e=\lambda\otimes S/\fm: k[G]^{\mu_e}\rightarrow k[G]^{\gamma_e}$.
  Then it is easy to see that $\lambda_e$ is 
  injective, and the cokernel of $\lambda_e$ is $S$-free.
  As $k[G]\otimes_k S$ is
  an injective object in the exact category of $S$-free modules in $\Cal C(G,S)$,
  it follows that $\lambda_e$ is a split monomorphism in $\Cal C(G,S)$.
  This shows that $\summ(k[G]\otimes_k S,{}^eL)\geq \mu_e$.

  So it suffices to prove that $\lim_{e\rightarrow\infty}\mu_e/p^{de}=\ell/g$.

  Let $\kappa$ denote the composite map $S^\ell\xrightarrow h L\rightarrow M$, and $Q:=\Coker\kappa$.
  Then
\begin{multline*}
  \lim_{e\rightarrow\infty}\mu_S(\Image {}^e\kappa)/p^{de}
  =
  \lim_{e\rightarrow\infty}\mu_S({}^eM)/p^{de}
  -
  \lim_{e\rightarrow\infty}\mu_S({}^eQ)/p^{de}\\
  =\eHK^S(M)-\eHK^S(Q)=\rank_S M-\rank_S Q=\rank_S S^\ell=\ell,
\end{multline*}
  where $\mu_S$ deontes the number of generators, and $\eHK^S$ denotes the Hilbert--Kunz multiplicity
  (as $S$ is the polynomial ring, $\eHK^S=\rank_S$).
  Note that
\[
  0\leq \mu_S(\Image{}^e\kappa)-\dim_k W_e\leq \rank_S{}^eM-\rank_S(k[G]\otimes_k S)^{\gamma_e}
  =p^{de}(m-g\gamma_e).
  \]
  So $\lim_{e\rightarrow \infty}(\mu_S(\Image{}^e\kappa)-\dim_k W_e)/p^{de}=0$.
  This shows $\lim_{e\rightarrow\infty}\dim_k W_e/p^{de}=\ell$.

  We can write ${}^eS=(k[G]\otimes S)^{\rho_e}\oplus Z_e$ such that $Z_e$ does not have a summand isomorphic
  to $k[G]$.
  Note that $\lim_{e\rightarrow\infty}\rho_e/p^{de}=1/g$.
  The restriction of $\delta_e:{}^eS^\ell\rightarrow W_e$ to $(k[G]\otimes S)^{\rho_e\ell}$ induces a $G$-linear map
  \[
  \eta_e: k[G]^{\rho_e\ell}\cong
  (k[G]\otimes S)^{\rho_e\ell}/\fm(k[G]\otimes S)^{\rho_e\ell}\rightarrow W_e.
  \]
  Let $Y_e$ be its image.
  Then $0\leq \dim_k W_e-\dim_k Y_e\leq\ell\rank_S Z_e=\ell (p^{de}-g\rho_e)$.
  So $\lim_{e\rightarrow\infty}\frac{1}{p^{de}}(\dim_k W_e-\dim_k Y_e)=0$.
  It follows that $\lim_{e\rightarrow\infty}\frac{1}{p^{de}}\dim_k Y_e=\ell$.
  Let $I_e$ be the kernel of $\eta_e$, and $J_e$ be its injective hull, which is a $G$-submodule of $k[G]^{\rho_e\ell}$.
  As $\dim J_e\leq g\dim I_e$, we have that $\dim J_e/p^{de}\rightarrow 0$.
  So when we set
\[
\tau_e:=\max\{r\mid \text{$k[G]^r$ is a submodule of $k[G]^{\rho_e\ell}/J_e$}\},
\]
we have that
  $\lim_{e\rightarrow\infty}\tau_e/p^{de}=\ell/g$.
  As $\tau_e\leq \mu_e\leq \dim W_e/g$, we have that $\lim_{e\rightarrow\infty}\mu_e/p^{de}=\ell/g$.
  This is what we wanted to prove.
\end{proof}

\begin{theorem}\label{main.thm}
  Let $k$ be a perfect field of characteristic $p>0$,
  $G$ be a finite $k$-group scheme over $k$,
  and $V$ a finite-dimensional $G$-module.
  Let $S=\Sym V$ be the symmetric algebra of $V$, and we assume that $S$ is graded so that
  each element of $V$ is homogeneous of degree one.
  Assume that the action of $G$ on $S$ is small.
  Let $M$ be a $\Bbb Q$-graded $S$-finite $S$-free $(G,S)$-module,
  and $L$ be its graded $(G,S)$-submodule which is reflexive as an $S$-module.
  Let $k=V_1,\ldots,V_r$ be the simple $G$-modules, and let $P_i$ be the projective cover of $V_i$.
  Then we have
\[
\FL(L^G)=\frac{\rank_S L}{\dim_k k[G]}[S'']=
\frac{\rank_S L}{\dim_k k[G]}
\sum_{i=1}^r \frac{\dim V_i}{\dim \End_G V_i} [(P_i\otimes_k S)^G],
\]
where $S''=(S\otimes_kk[G])^G$ is $S$ viewed as an $A$-module, where $A=S^G$.
\end{theorem}

\paragraph
We set $N:=G^\circ$ and $H:=G/N\cong G\red$.
We define $\psi:G\rightarrow G$ by the composite $G\rightarrow G/N=H=G\red\hookrightarrow G$.
For a $G$-module $W$, we define the $G$-module $\res_\psi W$ by $W'$, see (\ref{W'.par}).

Note that the coordinate algebra $k[N]$ is a $G$-module, or a right $k[G]$-comodule by the
coaction $k[N]\rightarrow k[N]\otimes k[G]$ induced by the map
$k[G]\rightarrow k[N]\otimes k[G]$ given by $f\mapsto \sum_{(f)}n_{(2)}\otimes (\Cal S n_{(1)})\cdot n_{(3)}\in k[N]\otimes k[G]$
(this map is well-defined, since $N$ is a normal subgroup scheme of $G$).
It is easy to see that for a $G$-module $W$, the coproduct
$\omega_W: W\rightarrow W'\otimes k[N]$ is $G$-linear.

\paragraph
We set $B=S^N$ and $A=S^G=B^H$.
By Lemma~\ref{infinitesimal-invariant-finite.lem},
we have that there exists some $e_0\geq 1$ such that $S^{p^{e_0}}\subset B$.

\paragraph
As $G$ is the semidirect product $G=N\rtimes H$, the composite
\[
k[G]\xrightarrow{\Delta_G}k[G]\otimes k[G]
\xrightarrow{\pi_N\otimes \pi_H}
k[N]\otimes k[H]'
\]
is an isomorphism, where $\pi_N:k[G]\rightarrow k[N]$
and $\pi_H:k[G]\rightarrow k[H]'$ are the canonical maps
corresponding to the canonical closed immersion $N\hookrightarrow G$ and $H\hookrightarrow G$, respectively.
Note that $B\subset S'\subset {}^{e_0}B$, and that
\begin{multline*}
  ({}^{e_0}(L^G)\otimes_A S)^{**}
  \cong (({}^{e_0}(L^G)\otimes_A B)^{\star\star}\otimes_B S)^{**}
  \cong ({}^{e_0}(L^N)\otimes_B S)^{**}\\
\cong
({}^{e_0}(L^N)\otimes_{S'} (S' \otimes_B S)^{**})^{**}
\cong
({}^{e_0}(L^N)\otimes_{S'} S'\otimes_k k[N])^{**}\\
\cong
({}^{e_0}(L^N)\otimes_k k[N])^{**}
\cong
F'\otimes_k k[N],
\end{multline*}
where $(-)^*=\Hom_S(-,S)$, $(-)^\star=\Hom_B(-,B)$, and $F={}^{e_0}(L^N)$.
Note that $F\in\Ref(H,S')$ and it is a graded $(H,S')$-submodule $M'$.
As $M'$ is finite free as an $S'$-module, we have that 
\[
\FL(F)=\frac{1}{|H|}\frac{p^{de_0}\rank_S L}{|N|}[S'\otimes_k k[H]]
=\frac{p^{de_0}\rank_SL}{|G|}[S'\otimes_k k[H]]
\]
in $\Theta^\circ(H,S')$ by Proposition~\ref{etale-case2.prop}, where $|H|=\dim_k k[H]$,
$|N|=\dim_k k[N]$, and $|G|=\dim_k k[G]$.

Note that for a $\Bbb Q$-graded $(H,S)$-module $Q$, 
$Q'\otimes_k k[N]$ is a $(G,S)$-module with the $S$-action
\[
s\cdot(q\otimes h)=\sum_{(s)}s_{(1)}q\otimes s_{(0)}h.
\]

\begin{lemma}\label{HNG.lem}
  $S'\otimes k[H]'\otimes k[N] $ is isomorphic to $S\otimes k[G]$ as a $(G,S)$-module.
\end{lemma}

\begin{proof}
  The maps in the sequence
  \[
  S\otimes k[G]\xrightarrow \Box S''\otimes k[G]\xrightarrow{\delta}S''\otimes k[H]'\otimes k[N]
  \xrightarrow{\Box_H^{-1}}S'\otimes k[H]'\otimes k[N]
\]
  are all isomorphisms, where
  $\Box(s\otimes f)=\sum_{(s)}s_{(0)}\otimes s_{(1)}f$,
  $\delta(s\otimes f)=\sum_{(f)}s\otimes \pi_H(f_{(1)})\otimes \pi_N(f_{(2)})$,
  and
  $\Box_H^{-1}(s\otimes h)=\sum_{(s)}s_{(0)}\otimes\Cal S_H(s_{(1)})h$,
where $\pi_H:k[G]\rightarrow k[H]$ and $\pi_N:k[G]\rightarrow k[N]$ are the canonical
surjective homomorphisms of $k$-Hopf algebras associated with the inclusions $H\hookrightarrow G$ and
$N\hookrightarrow G$, respectively.
These maps are isomorphisms of $G$-modules and isomorphisms of $k$-algebras.
Thus by $s(s'\otimes f)=ss'\otimes f$, $s(s''\otimes f)=\sum_{(s)}s_{(0)}s''\otimes s_{(1)}f$,
$s(s''\otimes h\otimes r)=\sum_{(s)}s_{(0)}s''\otimes s_{(1)}h\otimes s_{(2)}r$,
and
$s(s''\otimes h\otimes r)= s_{(0)}s''\otimes h\otimes s_{(1)}r$,
the $k$-algebras $S\otimes k[G]$, $S''\otimes k[G]$, $S''\otimes k[H]'\otimes k[N]$,
and $S'\otimes k[H]'\otimes k[N]$ are $G$-algebras, and the maps $\Box$, $\delta$, and $\Box_H^{-1}$
above are all isomorphisms of $(G,S)$-modules.
\end{proof}

\paragraph
Note that the $(H,S')$-module $F={}^{e_0}(M^N)$ is reflexive and is an $(H,S')$-submodule of ${}^{e_0}M'$,
which is finite free as an $S'$-module.
By Proposition~\ref{etale-case2.prop}, we have that
\[
\FL(F)=\frac{\rank_{S'}F}{|H|}[S'\otimes k[H]]
=
\frac{p^{de_0}\rank_SL}{|G|}[S'\otimes k[H]]
\]
in $\Theta^\circ(H,S')$.
Hence
\begin{multline*}
  \FL(L^G)=\frac{1}{p^{de_0}}\FL({}^{e_0}(L^G))= \frac{1}{p^{de_0}}\FL( (({}^{e_0}(L^G)\otimes_A S)^{**})^G)\\
=
\frac{1}{p^{de_0}}\FL((F'\otimes_k k[N])^G)  \\
=
\frac{\rank_SL}{|G|}[(S'\otimes_k k[H]'\otimes_k k[N])^G]
=\frac{\rank_SL}{|G|}[(S\otimes_k k[G])^G],
\end{multline*}
and Theorem~\ref{main.thm} has been proved.
\qed

\begin{corollary}\label{main.cor}
  Let the notation be as in Theorem~\ref{main.thm}.
  Then for every indecomposable $\Bbb Q$-graded finite $A$-module $N$,
  \[
  s(N,A)=\begin{cases}
  \dfrac{\dim V_i}{\dim k[G]\cdot \dim \End_G V_i} & (N\cong (P_i\otimes_k S)^G) \\
  0 & (\text{there is no \mt{i} such that $N\cong (P_i\otimes_k S)^G$})
  \end{cases}.
  \]
\end{corollary}

The following corollaries for the case that $G$ is linearly reductive was proved by
Watanabe--Yoshida \cite{WY}, Carvajal-Rojas--Schwede--Tucker \cite{CST}, and Carvajal-Rojas \cite{Carvajal-Rojas}.
For some important cases that $G$ is not linearly reductive was proved by 
Broer \cite{Broer}, Yasuda \cite{Yasuda},
and Liedtke--Martin--Matsumoto \cite{LMM}.

\begin{corollary}\label{main-cor2.cor}
  Let \mt{k} be a perfect field, \mt{V} a finite-dimensional \mt{k}-vector space,
  and \mt{G\subset\GL(V)} be a small finite subgroup scheme.
  Let \mt{S:=\Sym V}, and \mt{A:=S^G}.
  Then we have 
  \[
  s(A):=s(A,A)=
  \begin{cases}
  \dfrac{1}{\dim k[G]} & (\text{$G$ is linearly reductive}) \\
  0 & (otherwise)
  \end{cases}.
  \]
\end{corollary}

\begin{proof}
  Note that $A=S^G\cong (P_i\otimes_k S)^G$ up to degree shifting if and only if
  $S\cong P_i\otimes_k S$ up to degree shifting if and only if $k\cong P_i$ as $G$-modules.
  Considering the case that $N=A$ in Corollary~\ref{main.cor}, the corollary follows,
  since $G$ is linearly reductive if and only if $k\cong P_i$ for some $i$.
\end{proof}

\begin{corollary}\label{main-cor3.cor}
Let the notation be as in Corollary~\ref{main-cor2.cor}.
  The following are equivalent.
  \begin{enumerate}
  \item[\rm(1)] $G$ is linearly reductive.
  \item[\rm(2)]  Both $G^\circ$ and $G\red$ are linearly reductive.
  \item[\rm(3)] $\bar G^\circ\cong \mu_{p^{e_1}}\times\cdots\times\mu_{p^{e_r}}$ for some
    $e_1,\ldots,e_r\geq 1$, and
    $|\bar G\red|$ is not divisible by $p$, where $\bar G=\bar k\otimes_k G$ is the
    base change of $G$ to the algebraic closure $\bar k$ of $k$.
  \item[\rm(4)] $S^G$ is a direct summand subring of $S$.
  \item[\rm(5)] $S^G$ is strongly $F$-regular.
  \item[\rm(6)] $S^G$ is $F$-regular.
  \item[\rm(7)] $S^G$ is weakly $F$-regular.
  \item[\rm(8)] The $F$-signature $s(S^G)$ is positive.
  \end{enumerate}
\end{corollary}

\begin{proof}
  For the proof of (1)$\Leftrightarrow$(2)$\Leftrightarrow$(3), we may
  assume that $k$ is an algebraically closed field.
  
  For the equivalence (1)$\Leftrightarrow$(2), see \cite[Lemma~2.2]{Hashimoto2}.

  We prove (2)$\Rightarrow$(3).
  As $G^\circ$ is finite, connected and linearly reductive,
  we have $G^\circ\cong \mu_{p^{e_1}}\times\cdots\times\mu_{p^{e_r}}$ for some $e_1,\ldots,e_r\geq 1$ by
  \cite[Theorem~4.1]{Sweedler2}.
  Moreover, as $G\red$ is linearly reductive, $|G\red|$ is not divisible by $p$ by Maschke's
  theorem.

  (3)$\Rightarrow$(2).
  As $G^\circ=\mu_{p^{e_1}}\times\cdots\times\mu_{p^{e_r}}=\Spec k\Lambda$, where
  $\Lambda$ is the abelian group $\Bbb Z/p^{e_1}\Bbb Z\times\cdots\times \Bbb Z/p^{e_r}\Bbb Z$,
  the category of $G^\circ$-modules is equivalent to the category of $\Lambda$-graded $k$-vector spaces
  \cite[(I.2.11)]{Jantzen}, and hence $G^\circ$ is linearly reductive.
  On the other hand, $G\red$ is linearly reductive by Maschke's theorem.

  We prove (1)$\Rightarrow$(4).
  For a $G$-module $V$, let $U(V)=\sum_{S\subset V}S$, where the sum is taken over all the non-trivial simple $G$-submodules of $V$.
  Then we have a functorial decomposition $V=V^G\oplus U(V)$.
  The projection $V\rightarrow V^G$ is a $G$-linear splitting of the inclusion $V^G\rightarrow V$,
  and it is a Reynolds operator.
  So $S^G$ is a direct summand subring of $S$.

  (4)$\Rightarrow$(5).
  As $S$ is a polynomial ring over $k$, it is strongly $F$-regular.
  Hence its direct summand $S^G$ is also strongly $F$-regular, see \cite[Lemma~3.17]{Hashimoto}.

  For (5)$\Rightarrow$(6), see \cite[Corollary~3.7]{Hashimoto}.

  (6)$\Rightarrow$(7) is obvious by the definitions of the $F$-regularity and the weak $F$-regularity,
  see \cite[(4.5)]{HH}.

  (7)$\Rightarrow$(4).
  Let $A=S^G$, and $A^+$ be the integral closure of $A$ in the algebraic closure $K$ of the field of fractions $Q(A)$ of $A$.
  Then as in the proof of \cite[Proposition~2.14]{Smith}, we have that $IA^+\cap A\subset I^*$,
  where $I^*$ denotes the tight closure.
  By the definition of weak $F$-regularity, we have that $IA^+\cap A=I$.
  As $I$ is arbitrary, we have that $A\hookrightarrow A^+$ is cyclically pure.
  In particular, $A\hookrightarrow S$ is also cyclically pure.
  As $A$ is normal, we have that $A\rightarrow S$ is pure by \cite{Hochster}.
  By \cite[Corollary~5.3]{HR}, $A$ is a direct summand subring of $S$.

  The equivalence (5)$\Leftrightarrow$(8) is well-known,
  see \cite[Theorem~0.2]{AL}.

  (8)$\Rightarrow$(1) follows immediately from Corollary~\ref{main-cor2.cor}.
\end{proof}

\begin{flushleft}
Mitsuyasu HASHIMOTO\\
Department of Mathematics\\
Osaka Metropolitan University\\
Sumiyoshi-ku, Osaka 558--8585, JAPAN\\
e-mail: {\tt mh7@omu.ac.jp}
\end{flushleft}

\begin{flushleft}
Fumiya KOBAYASHI\\
Department of Mathematics\\
Osaka City University\\
Sumiyoshi-ku, Osaka 558--8585, JAPAN\\
e-mail: {\tt fullhouse235711@gmail.com}
\end{flushleft}

\end{document}